\documentclass[12pt,twoside]{amsart}
\usepackage{amssymb,amsmath,amsthm, amscd, enumerate, mathrsfs, upgreek}
\usepackage{graphicx, hhline, tikz}
\usepackage[colorlinks=true,pagebackref,hyperindex]{hyperref}
\usepackage[all]{xy}
\usepackage{color}
\usepackage[backrefs, alphabetic, initials]{amsrefs}
\usepackage{color} 
\usepackage{latexsym}
\usepackage[T1]{fontenc} 
\usepackage{fancyhdr}
\usepackage{enumerate}
\usepackage{array, longtable}
\usepackage{caption}
\usepackage{mathtools}
\newcolumntype{C}{>{$}c<{$}}
\newcolumntype{L}{>{$}l<{$}}

\title[Miyaoka type inequality]
{Miyaoka type inequality for terminal threefolds with nef anti-canonical divisors}

\date{\today, version 0.02}

\author{Masataka Iwai}
\address{Department of Mathematics, Graduate School of Science, Osaka University, 
Osaka 560-0043, Japan}
\email{{masataka@math.sci.osaka-u.ac.jp}, {masataka.math@gmail.com}}

\author{Chen Jiang}
\address{Fudan University, Shanghai Center for Mathematical Sciences \& School of Mathematical Sciences, Shanghai, 200438, China}
\email{chenjiang@fudan.edu.cn}

\author{Haidong Liu}
\address{Sun Yat-sen University, Department of mathematics, Guangzhou, 510275, China}
\email{liuhd35@mail.sysu.edu.cn, jiuguiaqi@gmail.com}

\DeclareMathOperator{\rank}{rank}

\DeclareMathOperator{\Sing}{Sing}

\DeclareMathOperator{\Aut}{Aut}

\newcommand\lcm{{\text{l.c.m.}}}

\newcommand{\xr}[1]{\textcolor{red}{#1}}

\newtheorem{thm}{Theorem}[section]
\newtheorem{lem}[thm]{Lemma}
\newtheorem{prop}[thm]{Proposition}

\newtheorem{cor}[thm]{Corollary}
\newtheorem{ques}[thm]{Question}

\theoremstyle{definition}
\newtheorem{ex}[thm]{Example}
\newtheorem{defn}[thm]{Definition}
\newtheorem{rem}[thm]{Remark}
\newtheorem*{ack}{Acknowledgments}

\newtheorem{case}{Case}
\newtheorem{claim}{Claim}

\makeatletter
 
 \@addtoreset{equation}{section}
\makeatother

\begin{document}

\begin{abstract}
In this paper, we study the Miyaoka type inequality on Chern classes of terminal projective $3$-folds with nef anti-canonical divisors. Let $X$ be a terminal projective $3$-fold such that $-K_X$ is nef. We show that if $c_1(X)\cdot c_2(X)\neq 0$, then $c_1(X)\cdot c_2(X)\geq \frac{1}{252}$; if further $X$ is not rationally connected, then $c_1(X)\cdot c_2(X)\geq \frac{4}{5}$ and this inequality is sharp. In order to prove this, we give a partial classification of such varieties along with many examples. We also study the nonvanishing of $c_1(X)^{\dim X-2}\cdot c_2(X)$ for terminal weak Fano varieties and prove a Miyaoka--Kawamata type inequality.
\end{abstract}

\maketitle 
\tableofcontents

\section{Introduction}\label{sec1}

In \cite{liu}*{Question 4.8}, the third author proposed the following question, which 
is a mirror version of Koll\'{a}r's expectation \cite{kollar}*{Remark 3.6} for minimal varieties.

\begin{ques}\label{ques.main}
Let $X$ be a terminal projective variety of dimension $n$ such that $-K_X$ is nef.
\begin{enumerate}
\item If $c_1(X)^{n-2} \cdot c_2(X) =0$, is there a classification of $X$?
\item If $c_1(X)^{n-2} \cdot c_2(X) >0$, is there a constant $\epsilon>0$
depending only on $n$ such that $c_1(X)^{n-2} \cdot c_2(X)\geq \epsilon $?
\end{enumerate} 
\end{ques}

This expectation heavily depends on Ou's result \cite{ou}*{Corollary 1.5} on the pseudo-effectivity of the second Chern classes, which implies that for a terminal projective variety $X$ of dimension $n$ such that $-K_X$ is nef,
\begin{align}
c_1(X)^{n-2}\cdot c_2(X)\geq 0.
\end{align} 
It can be viewed as a mirror version of Miyaoka's inequality \cite{miyaoka}*{Corollary 6.4} 
on the second Chern classes.

When $\dim X=2$, that is, when $X$ is a smooth projective surface such that $-K_X$ is nef, the answer to Question~\ref{ques.main} is clear and well-known. 
If $c_2(X)=0$, then it is well-known that $X$ admits a finite \'etale cover by an abelian surface or a 
$\mathbb P^1$-bundle over an elliptic curve (see \cite{ou}*{Theorem 1.8} and the references therein).
If $c_2(X)\neq 0$, then the Noether's formula $c_2(X)=12\chi(\mathcal{O}_X)-c_1(X)^2$ implies that $\chi(\mathcal{O}_X)\geq 1$ and hence $c_2(X)\geq 3$ as $c_1(X)^2\leq 9$; furthermore, if $X$ is not rational (or $-K_X$ is not big), then $c_1(X)^2=0$ and $c_2(X)\geq 12$.

The goal of this paper is to study Question~\ref{ques.main} in dimension 3, which is also motivated by the classification theory of $3$-folds.
For a terminal projective $3$-fold $X$ with $K_X\equiv 0$, $c_1(X)\cdot c_2(X)=0$ holds trivially. Such $3$-folds are known as ``Calabi--Yau $3$-folds" and have already been studied in many references (for example, \cites{kawamata, morrison, oguiso, op, jiang16}), so we will always assume that $K_X\not\equiv 0$. 
In this case, there are also many interesting results, see
 \cites{bp, cao-horing, xie1, bcs, matsumura-wang, lmptx, xie2} and the references therein.

Our main theorem is the following.

\begin{thm}\label{thm.main.bound}
Let $X$ be a terminal projective $3$-fold such that $-K_X$ is nef. Suppose that $c_1(X)\cdot c_2(X)\neq 0$. 
\begin{enumerate}
 \item Then $c_1(X)\cdot c_2(X)\geq \frac{1}{252}$. 
 \item If $-K_X$ is not big, then $c_1(X)\cdot c_2(X)\geq \frac{2}{5}$.
 \item If $X$ is not rationally connected, then $c_1(X)\cdot c_2(X)\geq \frac{4}{5}$.
\end{enumerate}
\end{thm}

 Theorem~\ref{thm.main.bound} is obtained by a partial classification of terminal projective $3$-folds with nef anti-canonical divisors.
 We briefly summarize the results in the following.

\begin{thm}\label{thm.roughclassify}
Let $X$ be a terminal projective $3$-fold such that $-K_X$ is nef and $K_X\not\equiv 0$. 
Then $X$ is one of the following:
\begin{enumerate}
 \item (Theorem~\ref{thm.smooth}, cf. \cite{bp}) $X$ is smooth and 
 \begin{enumerate}
 
 \item $X$ is a $\mathbb P^1$-bundle over an abelian surface or a bi-elliptic surface; 

 \item $X$ admits a locally trivial fibration over an elliptic curve, whose fibers are rationally connected; 
 
 \item $X$ is a product of $\mathbb P^1$ with a K3 surface; 

 \item $X$ is a $\mathbb P^1$-bundle over an Enriques surface,
	which is trivialized by the universal double \'etale cover; 

 \end{enumerate}

 \item $X$ is rationally connected;
 
 \item (Theorem~\ref{thm:X=X'/G}) $X$ is singular and not rationally connected, and $X=X'/G$ where $X'$ is either a $\mathbb{P}^1$-bundle over an abelian surface or a product of $\mathbb{P}^1$ with a K3 surface, and $G$ is a group acting on $X'$ which can be classified. 
\end{enumerate} 
\end{thm}
From this classification, we can describe the structure of $X$ when $c_1(X)\cdot c_2(X)=0$.

\begin{cor}\label{cor.roughclassify=0}
Let $X$ be a terminal projective $3$-fold such that $-K_X$ is nef and $K_X\not\equiv 0$. 
Then $c_1(X)\cdot c_2(X)=0$ if and only if $X$ is one of the following:
\begin{enumerate}
 \item $X$ admits a quasi-\'etale cover by a $\mathbb P^1$-bundle over an abelian surface; 

 \item $X$ admits a locally trivial fibration over an elliptic curve, whose fibers are rationally connected; 
 
 \item $X$ is rationally connected, $K_X^3=0$, and the set $\mathcal{R}_X$ of local indices of the virtue singularities (see \S\ref{sec:hec} for the definition) belongs to 
 Table~\ref{tab1} of Theorem~\ref{thm.list} except types $(5), (6)$ and $(10)$.
 
\end{enumerate} 
\end{cor}

\begin{rem}The lower bound of Theorem~\ref{thm.main.bound}(3) is sharp by Example~\ref{ex.3}, while the lower bounds of Theorem~\ref{thm.main.bound}(1)(2) might not be sharp. The main difficulty is the case when $X$ is non-Gorenstein and rationally connected, in which case there are no good methods to classify $X$ in a more explicit way. For example, one natural strategy in birational geometry is to run a minimal model program, and then describe each step of the minimal model program and classify the outcomes (which are Mori fiber spaces in this case). But so far neither of these can be done in an explicit way as singularities are involved. Also, we will lose the nefness of $-K_X$ when running the minimal model program.
\end{rem}

Finally, we focus on the case when $-K_X$ is nef and big. Such $X$ is usually called a \emph{weak Fano} variety; and $X$ is called a \emph{Fano} variety if $-K_X$ is ample. Terminal weak Fano $3$-folds have been also studied in many references (see \cites{cc, jiang-zou} and the references therein). 
In fact, $c_1c_2>2$ for the majority of known examples of terminal Fano $3$-folds (see Example~\ref{ex.fano}). So it is interesting to improve the lower bound of $c_1c_2$ for terminal weak Fano $3$-folds or to find new examples with small 
$c_1c_2$. 

Another interesting problem is whether there exist terminal weak Fano $3$-folds 
with $c_1c_2=0$. In this direction, we prove that this can not happen.

\begin{thm}[Corollary~\ref{cor.fano}]\label{thm.main.fano}
Let $X$ be a terminal weak Fano variety of dimension $n$.
Then $c_1(X)^{n-2}\cdot c_2(X)>0$.
\end{thm}

As a corollary, we can show a Miyaoka--Kawamata type inequality (see \cites{kawamata92, bk} for the case $\rho(X)=1$).
\begin{cor}[Corollaries~\ref{cor.chernbound} and~\ref{cor.effbd}]\label{cor.effbd1}
There exists a positive rational number $b$ depending only
on $n$ such that the inequality 
\[
c_1(X)^n\leq b \cdot c_1(X)^{n-2}\cdot c_2(X)
\]
holds for any terminal weak Fano variety $X$ of dimension $n$.
In particular, we can choose $b=2^4\cdot 3^6\cdot 7$ when $n=3$.
\end{cor}

\begin{rem}
The universal upper bound $b= 2^4\cdot 3^6\cdot 7$ in Corollary~\ref{cor.effbd1} is far from being sharp. 
We would conjecture that $b$ should be 3, the same as the constant in \cite{miyaoka}*{Theorem 1.1} (see Remark~\ref{rem.b=3}).
\end{rem}

\begin{rem}
Recently, the third author and Jie Liu proved in \cite{liu-liu} that
for terminal $\mathbb Q$-Fano varieties, i.e., $\mathbb Q$-factorial terminal Fano varieties with Picard number one, 
the bound $b$ in Corollary \ref{cor.effbd1} can be improved to be $4$ (and to be $25/8$ in dimension 3).
\end{rem}

\begin{ack}
The authors would like to thank Professors Vladimir Lazi\'c,
Shin-ichi Matsumura, Thomas Peternell and Shilin Yu for useful discussions
and suggestions. 
The first author would like to thank Kento Fujita for useful comments. 
The third author would like to thank Xiaobin Li for his computer support. We thank the referees for useful comments and suggestions.

This work was supported by National Natural Science Foundation of China for Innovative Research Groups (Grant No. 12121001) and National Key Research and Development Program of China (Grant No.~2020YFA0713200).
The first author was supported by Grant-in-Aid for Early Career Scientists $\sharp$ 22K13907.
The second author is a member of LMNS, Fudan University.
\end{ack}

Throughout this paper, we work over the complex number field $\mathbb C$.
We will freely use the basic notation in \cites{kollar-mori}.

\section{A structure theorem of Cao--H\"oring and Matsumura--Wang}

Let $X$ be a terminal projective variety such that $-K_X$ is nef.
There exists a structure theorem of $X$
by \cite{cao-horing}*{Theorem 1.3}
for smooth cases and \cite{matsumura-wang}*{Theorem 1.1} for klt cases. We recall these results in the setting of terminal singularities as follows.

\begin{defn}
    Let $f\colon X \to Y$ be a proper surjective morphism between normal varieties with connected fibers. Then $f$ is called a \emph{locally constant fibration} if $f$ is a locally trivial fiber bundle with a representation
    \[
        \rho\colon \pi_1(Y) \to \Aut(F),
    \]
    where $\pi_1(Y)$ is the fundamental group of $Y$ and $F$ is the fiber of $f$, such that $X$ is isomorphic to the quotient of $\widetilde Y\times F$ by the action of $\pi_1(Y)$ given by $\delta\cdot (\widetilde y,z)=(\delta\cdot \widetilde y, \rho(\delta)(z))$ where $\widetilde Y$ is the universal cover of $Y$.
\end{defn}

\begin{thm}[{cf. \cites{cao-horing, wang, matsumura-wang}}]\label{thm.structure}
Let $X$ be a terminal projective variety such that $-K_X$ is nef. Then,
there exists a finite cover $\pi\colon X'\to X$ \'etale in codimension $2$ 
such that $X'$ admits a holomorphic MRC fibration $f\colon X'\to Y$
satisfying the following properties:
\begin{enumerate}[(1)]
 \item $f$ is locally constant;

 \item any fiber $F$ of $f$ is a rationally connected variety with terminal singularities;

 \item $Y$ is a terminal projective variety with $K_Y\sim_{\mathbb Q} 0$.
\end{enumerate}
Moreover, if $X$ is smooth, 
then we may choose $\pi$ to be the identity map and $Y$ is smooth. 
\end{thm}

\begin{proof}
The existence of $\pi$ and $f$ is by \cite{matsumura-wang}*{Theorem 1.1} and the last statement is by \cite{cao-horing}*{Theorem 1.3}.

Here note that by \cite{matsumura-wang}*{Theorem 1.1}, $\pi$ is \'etale in codimension $1$ and $Y, F$ are klt. So we shall explain why $\pi$ is \'etale in codimension $2$ and $Y, F$ are terminal in our setting. 

Since $X$ is terminal, the singular locus $\Sing X$ is of codimension at least 3 by \cite{kollar-mori}*{Corollary~5.18}.
Since $\pi$ is \'etale in codimension $1$,
$\pi$ is actually \'etale over the smooth locus $X\setminus \Sing X$ by the purity of ramification locus
(Zariski's purity theorem). Therefore, 
$\pi\colon X'\to X$ is \'etale in codimension $2$.

Since $X$ is terminal, $X'$ is also terminal by \cite{kollar-mori}*{Proposition~5.20};
then $Y, F$ are also terminal since $f$ is locally constant.
\end{proof}

\begin{lem}\label{lem:rc=0}
Let $F$ be a rationally connected projective variety with rational singularities. Then $H^j(F, \mathcal O_F)=0$ for $j\geq 1$. In particular, $\chi(\mathcal{O}_F)=1$.
\end{lem}

\begin{proof}
As $F$ has only rational singularities, it suffices to show that for a resolution $F'$ of $F$, $H^j(F', \mathcal O_{F'})=0$ for $j\geq 1$. Note that $F'$ is again rationally connected, so the conclusion follows from
\cite{kollar-book}*{Corollary~IV.3.8}.
\end{proof}

\begin{prop}\label{prop:fib hi=hi}
Let $f\colon X'\to Y$ be the fibration in Theorem~\ref{thm.structure}. Then $h^i(X',\mathcal O_{X'})=h^i(Y,\mathcal O_Y)$ 
for each $i\geq 0$. In particular, $\chi(\mathcal O_{X'})=\chi(\mathcal O_Y)$.
\end{prop}

\begin{proof}
Since any fiber $F$ of $f$ is rationally connected and terminal,
 $H^j(F, \mathcal O_F)=0$ for $j\geq 1$ by Lemma~\ref{lem:rc=0}. In particular,
$R^jf_*\mathcal O_{X'}=0$ for $j\geq 1$ as the fibration is locally constant. Then $h^i(X',\mathcal O_{X'})=h^i(Y,\mathcal O_Y)$ by the Leray spectral sequence. 
\end{proof}

\begin{lem}\label{lem:q>0 smooth}
Let $X$ be a terminal projective $3$-fold such that $-K_X$ is nef. If the irregularity $q(X):=h^1(X,\mathcal O_X)>0$, then $X$ is smooth. 
\end{lem}

\begin{proof}
By \cite{wang}*{Theorem A} or \cite{matsumura-wang}*{Corollary 4.9}, the Albanese map of $X$ is locally constant and surjective. Any fiber of the Albanese map of $X$ has dimension at most $2$ and has terminal singularities, so it is smooth, which implies that $X$ is smooth.
\end{proof}

If a normal variety $X$ is smooth in codimension $k$, then by extending from the smooth locus, the Chern classes $c_i(X)$ can be defined as elements sitting in the Chow groups $A^i(X)$ for $i\leq k$. In particular, $c_1(X)$ and $c_2(X)$ are well-defined for terminal varieties. 
The following lemma for Chern classes is easy but useful.

\begin{lem}\label{lem.chern}
Let $X$ be a normal projective $\mathbb{Q}$-Gorenstein variety of dimension $n$ which is smooth in codimension $2$. Let $\pi\colon X'\to X$ be a finite morphism which is
 \'etale in codimension $2$, then $c_i(X')=\pi^*c_i(X)$ for $i=1, 2$.
In particular, $c_1(X')^{n-2} \cdot c_2(X')=\deg \pi \cdot
c_1(X)^{n-2} \cdot c_2(X)$.
\end{lem}

\section{Reid's formula and holomorphic Euler characteristics}\label{sec:hec}

A {\it basket} $B$ is a collection of pairs of coprime integers (permitting weights). 
Let $X$ be a terminal projective $3$-fold. According to Reid~\cite{reid}, 
there is a basket of virtual orbifold points
\[
B_{X}=\{(b_{i}, r_{i}) \mid i=1, \cdots, s ; 0<b_{i}\leq\frac{r_{i}}{2} ; b_{i} \text{ is coprime to } r_{i}\}
\]
associated to $X$, where a pair $\left(b_{i}, r_{i}\right)$ corresponds to an orbifold point $Q_{i}$ of type $\frac{1}{r_{i}}\left(1, -1, b_{i}\right)$. 
Here recall that every singularity of $X$, as a $3$-dimensional terminal singularity, can be locally deformed into a set of cyclic quotient singularities, and these cyclic quotient singularities are called {\it virtual orbifold points} of $X$.
Denote by $\mathcal{R}_X$ the collection of $r_i$ (permitting weights) appearing in $B_X$ or as a set of integers with weights where weights appear in superscripts, say for example, 
\[
\mathcal{R}_X=\{3, 3,7,11\}=\{ 3^2, 7,11\}.
\]
Note that the Cartier index $r_X$ of $K_X$ is just $\lcm\{r_i\mid r_i\in \mathcal{R}_X\}$.

According to Reid~\cite{reid}*{10.3}, 
for any positive integer $n$, 
\begin{align}\label{eq:Reid-RR}
\chi(-nK_X)&=\frac{1}{12} n(n+1)(2 n+1)\left(-K_{X}^{3}\right)+(2 n+1)\chi(\mathcal{O}_X)-l(n+1) 
\end{align}
where $l(n+1)=\sum_{i} \sum_{j=1}^{n} \frac{\overline{j b_{i}}\left(r_{i}-\overline{j b_{i}}\right)}{2 r_{i}}$ and the first sum runs over $B_X$. Here $\overline{jb_i}$ means the smallest non-negative residue of $jb_i \bmod r_i$.


In this case, $\chi(\mathcal O_X)$ is related to $c_1(X)\cdot c_2(X)$ and the singularities
 as follows:

\begin{thm}[\cites{kawamata, reid}]\label{thm.euler}
Let X be a terminal projective $3$-fold. Then
\begin{equation}\label{eq.range}
\chi(\mathcal O_X)=\frac{1}{24}c_1(X)\cdot c_2(X)+\frac{1}{24}\sum_{r_i\in \mathcal{R}_X} (r_i-\frac{1}{r_i}).
\end{equation}
\end{thm}



As a corollary, we get the finiteness of singularities of terminal $3$-folds with nef anti-canonical divisors.
\begin{cor}\label{cor.bound}
Let $X$ be a terminal projective $3$-fold $X$ such that $-K_X$ is nef and $K_X\not \equiv 0$. Then the following hold:
\begin{enumerate}
 \item $ \chi(\mathcal O_X)\in \{0,1,2\}$. Moreover, $ \chi(\mathcal O_X)\neq 0$ if $X$ is singular.
 
 \item There are only finitely many possibilities for $\mathcal{R}_X$, $B_X$, and $c_1(X)\cdot c_2(X)$.
\end{enumerate}
\end{cor}

\begin{proof}
(1) We may assume that $X$ is not rationally connected by Lemma~\ref{lem:rc=0}.

Let $f\colon X'\to Y$ be the fibration in Theorem~\ref{thm.structure}. Then $\dim Y\in \{1, 2\}$ as $K_X\not\equiv 0$ and $X$ is not rationally connected. Moreover, $Y$ is smooth and $K_Y\sim_\mathbb{Q} 0$. Hence $\chi(\mathcal{O}_{X'})=\chi(\mathcal{O}_Y)\in \{0,1,2\}$ by Proposition~\ref{prop:fib hi=hi} and the classical classification of curves and surfaces. 

Let $q(X):=h^1(X,\mathcal O_X)$ be the irregularity of $X$. If $q(X)>0$, then $X$ is smooth by Lemma~\ref{lem:q>0 smooth}. Hence
$\chi(\mathcal{O}_{X})=\chi(\mathcal{O}_{X'})\in \{0,1,2\}$ by Theorem~\ref{thm.structure}.

If $q(X)=0$ (which contains the case that $X$ is singular), then as $X'\to X$ is finite, $$h^2(X, \mathcal{O}_X)\leq h^2(X', \mathcal{O}_{X'})=h^2(Y, \mathcal{O}_{Y})\leq 1$$ by Proposition~\ref{prop:fib hi=hi} and the classical classification of curves and surfaces. Also $h^3(X, \mathcal{O}_X)=0$ as $K_X\not\equiv 0$. 
Hence $$\chi(\mathcal{O}_X)=1+h^2(X, \mathcal{O}_X)\in \{1, 2\}.$$

(2) By \cite{ou}*{Corollary 1.5}, $c_1(X)\cdot c_2(X)\geq 0$. Hence by Theorem~\ref{thm.euler}, as 
\begin{align}
 \sum_{r_i\in \mathcal{R}_X} (r_i-\frac{1}{r_i})\leq 24\chi(\mathcal O_X)\leq 48, \label{eq:ri<chi}
\end{align}
there are only finitely many possibilities for $\mathcal{R}_X$ and hence also for $B_X$.
Therefore, by Theorem~\ref{thm.euler} again, there are only finitely many possibilities for $c_1(X)\cdot c_2(X)$. 
\end{proof}

If $-K_X$ is nef but not big, then there is a strong restriction on $B_X$.
\begin{lem}\label{lem.lint}
Let $X$ be a terminal projective $3$-fold such that 
$-K_X$ is nef but not big. Then $l(n+1)$ is an integer for any positive integer $n$. In particular, $l(2)=\sum_{i} \frac{ b_{i}(r_{i}-b_{i})}{2 r_{i}}$ is an integer.
\end{lem}
\begin{proof}
Since $(-K_X)^3=0$, by \eqref{eq:Reid-RR}, $l(n+1)=(2n+1)\chi(\mathcal O_X)-\chi(-nK_X)$ is an integer. 
\end{proof}

\section{The case $\chi=1$}
In this section, we study the case when $X$ is a terminal projective $3$-fold such that $-K_X$ is nef and $\chi(\mathcal O_X)=1$. In particular, it contains the 
case that $X$ is rationally connected by Lemma~\ref{lem:rc=0}. 

\begin{thm}\label{thm.list}
Let $X$ be a terminal projective $3$-fold such that $-K_X$ is nef and $\chi(\mathcal O_X)=1$.
Then the following statements hold.
\begin{enumerate}
 \item Either $ c_1(X)\cdot c_2(X)\geq \frac{1}{252}$ or $c_1(X)\cdot c_2(X)=0$, where in the latter case, $\mathcal{R}_X$ is one of the following types:
{
\begin{longtable}{LLL}
\caption{$\mathcal{R}_X$ for $\chi=1$ and $c_1c_2=0$}\label{tab1}\\
\hline
\text{No.} & \mathcal{R}_X & r_X \\
\hline
\endfirsthead
\\
\hline 
\text{No.} & \mathcal{R}_X & r_X \\
\endhead
\hline
\endfoot

\hline \hline
\endlastfoot

1& \{ 5^5\} & 5\\
2& \{ 2^3, 4, 8^2\} & 8\\
3& \{ 2^3, 5^2,10\} & 10 \\
4& \{ 2^2, 3^2,4,12\} & 12 \\
5& \{ 2, 4^6\} & 4\\
6& \{ 3^4, 4^2,6\} & 12 \\
7& \{3^9 \} & 3\\
8& \{2^6,4^4\} & 4 \\
9& \{2^5,3^4,6\} & 6\\
10& \{2^{11},4^2\} & 4 \\
11& \{2^{16}\} & 2\\
\hline
\end{longtable}
}

\item If moreover $-K_X$ is not big, then $ c_1(X)\cdot c_2(X)\geq \frac{2}{5}$ or $c_1(X)\cdot c_2(X)=0$; and $\mathcal{R}_X$ is one of the following types:
{
\begin{longtable}{LLLL}
\caption{$\mathcal{R}_X$ for $\chi=1$ and $-K$ not big}\label{tab2}\\
\hline
\text{No.} & \mathcal{R}_X & r_X & c_1c_2 \\
\hline
\endfirsthead
\multicolumn{4}{l}{{ {\bf \tablename\ \thetable{}} \textrm{-- continued}}}
\\
\hline 
\text{No.} & \mathcal{R}_X & r_X & c_1c_2 \\
\endhead
\hline
\hline \multicolumn{4}{c}{{\textrm{Continued on next page}}} \\ \hline
\endfoot

\hline \hline
\endlastfoot
 1 & \{5^2\} & 5 & 72/5\\
 2 & \{2,3,6\} & 6 & 14 \\
 3 & \{2,4^2\} & 4 & 15 \\
 4 & \{3^3\} & 3 & 16 \\
 5 & \{7^3\} & 7 & 24/7 \\
 6 & \{2^4\} & 2 & 18\\
 7 & \{2^2,10^2\} & 10 & 6/5\\
 8 & \{2,4,8^2\} & 8 & 3\\
 9 & \{5^4\} & 5 & 24/5\\
 10 & \{2,3,5^2,6\} & 30 & 22/5\\
 11 & \{2,4^2,5^2\} & 20 & 27/5\\
 12 & \{3^3,5^2\} & 15 & 32/5\\
 13 & \{ 5^5\} & 5 & 0 \\
 14 & \{2^4,5^2\} & 10 & 42/5\\
 15 & \{ 2^3, 4, 8^2\} & 8 & 0 \\
 16 & \{ 2^3, 5^2,10\} & 10 & 0 \\
 17 & \{2^3,6^3\} & 6 & 2\\
 18 & \{ 2^2, 3^2,4,12\} & 12 & 0 \\
 19 & \{2^2,3^2,6^2\} & 6 & 4 \\
 20 & \{2^2,3,4^2,6\} & 12 & 5\\
 21 & \{2^2,4^4\} & 4 & 6\\
 22 & \{2,3^4,6\} & 6 & 6\\
 23 & \{2,3^3,4^2\} & 12 & 7\\
 24 & \{3^6\} & 3 & 8\\
 25 & \{2^5,3,6\} & 6 & 8 \\
 26 & \{2^5,4^2\} & 4 & 9\\
 27 & \{2^4,3^3\} & 6 & 10 \\
 28 & \{2^8\} & 2 & 12\\
 29 & \{2^4,3^3,5^2\} & 30 & 2/5\\
 30 & \{3^9 \} & 3& 0 \\
 31 & \{2^8,5^2\} & 10 & 12/5\\
 32 & \{2^6,4^4\} & 4 & 0 \\
 33 & \{2^5,3^4,6\} & 6 & 0\\
 34 & \{2^5,3^3,4^2\} & 12 & 1\\
 35 & \{2^4,3^6\} & 6 & 2 \\
 36 & \{2^9,3,6\} & 6 & 2\\
 37 & \{2^{9},4^2\} & 4& 3\\
 38 & \{2^8,3^3\} & 6 & 4\\
 39 & \{2^{12}\} & 2 & 6\\
 40 & \{2^{16}\} & 2 & 0\\
\hline
\end{longtable}
}
\end{enumerate}
\end{thm}

\begin{proof}
(1) Note that 
\begin{align}
\sum_i(r_i-\frac{1}{r_i})\leq 24\chi({\mathcal{O}_X})=24\label{eq:ri<24}
\end{align}
by \eqref{eq:ri<chi}
and 
\begin{align*}
c_1(X)\cdot c_2(X)=24-\sum_i(r_i-\frac{1}{r_i}).
\end{align*}
With the aid of a computer program, we can list all possible $\mathcal{R}_X$ satisfying \eqref{eq:ri<24}. Among them, there are $11$ possibilities of $\mathcal{R}_X$ as listed in Table~\ref{tab1} with $c_1(X)\cdot c_2(X)=0$, while $c_1(X)\cdot c_2(X)\geq \frac{1}{4\cdot 7\cdot 9}= \frac{1}{252}$ for the remaining ones. Here the minimal value is obtained by $\mathcal{R}_X=\{2^{3},4,7,9\}$. 

(2) As $-K_X$ is not big, $l(2)\in \mathbb Z$ by Lemma~\ref{lem.lint}, where
\[
l(2)=\sum_{i} \frac{ b_{i} (r_{i}- b_{i})}{2 r_{i}}.
\]
With the aid of a computer program, we can list all possible $\mathcal{R}_X$ satisfying \eqref{eq:ri<24} such that $l(2)\in \mathbb Z$ for some basket $B_{X}=\{(b_{i}, r_{i})\}$, as in Table~\ref{tab2}. Here the minimal value is obtained by 
$\mathcal{R}_X=\{2^4,3^3,5^2\}$ as No.29 of Table~\ref{tab2}.
\end{proof}



The following example provides a rationally connected terminal projective $3$-fold such that $-K_X$ is nef but not big with $c_1(X)\cdot c_2(X)=0$.
\begin{ex}\label{ex.4}
Let $E$ be an elliptic curve and
$$Q=\{[x_1:x_2:x_3:x_4]\in \mathbb P^3\mid x^2_1+x^2_2+x^2_3+x^2_4=0\}$$
be a smooth quadric surface in $\mathbb P^3$. Let $X'=E\times Q$. 
It is easy to see that $-K_{X'}$ is nef with $(-K_{X'})^3=0$,
and $c_1(X')\cdot c_2(X')=0$. 
Consider the involution $\tau$ on $X'$:
\[
 (e; [x_1:x_2:x_3:x_4])\to (-e; [-x_1:-x_2:x_3:x_4]).
\]
Let $X\coloneqq X'/\tau$ and $f\colon X\to E/-1\simeq \mathbb P^1$ be the natural projection. Then $\pi\colon X'\to X$ is \'etale in codimension $2$
and $X$ is a $\mathbb Q$-factorial terminal $3$-fold with exactly 16 points of type $\frac{1}{2}(1,1,1)$.
Note that $X$ is rationally connected by \cite{ghs}*{Corollary 1.3},
$-K_X$ is nef but not big, and $c_1(X)\cdot c_2(X)=0$ by Lemma~\ref{lem.chern}. 
In this case, $\mathcal{R}_X=\{2^{16}\}$ as No.11 of Table~\ref{tab1} (or No.40 of Table~\ref{tab2}). 
\end{ex}

\begin{rem}\label{rem.chi2}
In the case $\chi(\mathcal O_X)=2$, we can still use computer to list all the possibilities of $\mathcal{R}_X$, which are numerous (over 1000 candidates). In \S\ref{sec.nrcsc}, we will use a more geometric method to deal with this case. 
\end{rem}


\section{Smooth cases}\label{sec.sc}

For a smooth projective $3$-fold $X$ such that $-K_X$ is nef, a detailed classification has been given by 
Bauer--Peternell in \cite{bp}. Combining
with Cao--H\"oring's recent result \cite{cao-horing},
we summarize the result as the following.

\begin{thm}\label{thm.smooth}
Let $X$ be a smooth projective $3$-fold such that $-K_X$ is nef and $K_X\not \equiv 0$. Let $q(X):=h^1(X,\mathcal O_X)$ be the irregularity of $X$.
Then $X$ is one of the following cases:
\begin{enumerate}

 \item \label{No/ab} $X$ is a $\mathbb P^1$-bundle over an abelian surface; 

 \item \label{No/bi} $X$ is a $\mathbb P^1$-bundle over a bi-elliptic surface; 

 \item \label{No/ell} $X$ admits a locally trivial fibration over an elliptic curve whose fibers are rationally connected; 
 
 \item \label{No/K3} $X=\mathbb P^1\times Y$, where $Y$ is a K3 surface; 

 \item \label{No/Enr} $X$ is a $\mathbb P^1$-bundle over an Enriques surface,
	which is trivialized by the universal double \'etale cover; 
	
 \item \label{No/RC} $X$ is rationally connected. 
 
\end{enumerate}

{
\begin{longtable}{CCCCCCC}
\caption{Invariants of $X$}\label{tab3}\\
\hline
\text{No.} & c_2(X) & q(X) & h^2(\mathcal{O}_X) & h^3(\mathcal{O}_X) & \chi(\mathcal{O}_X) &c_1c_2 \\
\hline
\endfirsthead
\multicolumn{4}{l}{{ {\bf \tablename\ \thetable{}} \textrm{-- continued}}}
\\
\hline 
\text{No.} & c_2(X) & q(X) & h^2(\mathcal{O}_X) & h^3(\mathcal{O}_X) & \chi(\mathcal{O}_X) &c_1c_2 \\
\endhead
\hline
\hline \multicolumn{4}{c}{{\textrm{Continued on next page}}} \\ \hline
\endfoot

\hline \hline
\endlastfoot
\eqref{No/ab} & 0 & 2 & 1 & 0 & 0 & 0\\
\eqref{No/bi} & 0 & 1 & 0 & 0 & 0 & 0\\
\eqref{No/ell} & \neq 0 & 1 & 0 & 0 & 0 & 0\\
\eqref{No/K3} & \neq 0 & 0 & 1 & 0 & 2 & 48\\
\eqref{No/Enr} & \neq 0 & 0 & 0 & 0 & 1 & 24\\
\eqref{No/RC} & \neq 0 & 0 & 0 & 0 & 1 & 24\\
\hline
\end{longtable}
}

\end{thm}

\begin{proof}
By Theorem~\ref{thm.structure},
there exists a locally constant fibration
$f\colon X\to Y$ such that any fiber $F$ of $f$ is rationally connected and $Y$ is a smooth projective variety with $K_Y\sim_{\mathbb Q} 0$. 
Since $K_X\not \equiv 0$, $\dim Y\in \{0,1,2\}$.

\begin{case}
If $\dim Y=2$, then the fiber $F$ is $\mathbb P^1$.
Since $K_Y\sim_{\mathbb Q} 0$, 
$Y$ belongs to one of the $4$ classes of minimal surfaces with Kodaira dimension $0$
by the Enriques--Kodaira classification.

If $Y$ is an abelian surface or a bi-elliptic surface, then we get Case \eqref{No/ab} or \eqref{No/bi}. 
From the exact sequence
\[
0\to \mathcal T_{X/Y}\to \mathcal T_X\to f^*\mathcal T_Y\to 0,
\]
we have
\[
c_2(X)=c_2(\mathcal T_{X/Y})+c_1(\mathcal T_{X/Y})\cdot f^*c_1(Y)+f^*c_2(Y)=0.
\]
Here we used the facts that $\mathcal T_{X/Y}$ is a line bundle as $\dim F=1$, $c_1(Y)=0$, and $c_2(Y)=0$.

If $Y$ is a K3 surface or an Enriques surface, then $q(X)=h^1(X,\mathcal O_X)=h^1(Y,\mathcal O_Y)=0$ by Proposition~\ref{prop:fib hi=hi}.
Then by \cite{bp}*{Corollary 4.4 and Corollary 3.2}, we get Case \eqref{No/K3} or \eqref{No/Enr}. 
\end{case}

\begin{case}
If $\dim Y=1$, then $Y$ is an elliptic curve and we get Case \eqref{No/ell}. 
From the exact sequence 
\[
0\to \mathcal{T}_F \to \mathcal{T}_X|_F\to \mathcal{N}_{F/X}\simeq \mathcal{O}_F(F)\to 0,
\]
we have
\[
c_2(X)\cdot F= c_2(F)=e(F)=2+h^{1,1}(F)>0,
\]
 which implies that $c_2(X)\neq 0$. Here we used the fact that $F$ is a rationally connected surface.
\end{case}
\begin{case}
If $\dim Y=0$, then $X$ is rationally connected and we get Case \eqref{No/RC}.
\end{case}
\setcounter{case}{0}

By Proposition~\ref{prop:fib hi=hi}, $h^i(X,\mathcal O_{X})=h^i(Y,\mathcal O_Y)$ 
for each $i\geq 0$. 
By the classical Hirzebruch--Riemann--Roch theorem, $24\chi(\mathcal{O}_X)=c_1(X)\cdot c_2(X).$ So to complete Table~\ref{tab3}, we only need to compute $c_2(X)$ and $h^i(Y,\mathcal O_Y)$ 
for each $i\geq 0$, which are clear from the above discussions and the classification of $Y$. 
\end{proof}

\begin{rem}\label{rem.-nef1}
In Cases~\eqref{No/ab} and~\eqref{No/bi} of Theorem~\ref{thm.smooth}, 
the second Chern class $c_2(X)$ is trivial.
In Case~\eqref{No/ab}, after a finite \'etale cover, we can write $X=\mathbb P(E)$ 
where $E$ is a numerically flat bundle of rank 2 (see \cite{bp}*{Remark 4.3});
for Case~\eqref{No/bi}, after an \'etale base change, it can be reduced to Case~\eqref{No/ab}.
\end{rem}

\begin{rem}\label{rem.-nef2}
While the equality $c_1(X)\cdot c_2(X)=0$ holds trivially in 
Cases~\eqref{No/ab} and~\eqref{No/bi} of Theorem~\ref{thm.smooth},
we have shown that neither $c_1(X)$ nor $c_2(X)$ is trivial in Case~\eqref{No/ell}.
Since the fiber $F$ is a smooth rationally connected surface such that $-K_F=-K_X|_F$ is nef, we have a complete classification of these surfaces (see \cite{bp}*{Propositions 1.5 and 1.6} and the last paragraph of \cite{bp}*{Preliminaries}). 

On the other hand, there are also classification results for rationally connected $3$-folds with nef anti-canonical divisors, see for example \cites{bp, xie1, xie2}. 

 
\end{rem}


\section{Non-rationally connected singular cases}\label{sec.nrcsc}

In this section, we give a classification of non-rationally connected singular terminal projective $3$-folds with nef and non-trivial anti-canonical divisors.

\begin{lem}\label{lem:nonrc-Y-not-curve}
Let $X$ be a terminal projective $3$-fold such that $-K_X$ is nef and $K_X\not \equiv 0$. Assume that $X$ is singular and not rationally connected. Let $f\colon X'\to Y$ be the fibration in Theorem~\ref{thm.structure}. Then $\dim Y= 2$.
\end{lem}

\begin{proof}
If $\dim Y=0$, then $X'$ is rationally connected, which contradicts to the assumption.

If $\dim Y=1$, then consider the MRC fibration $X\dashrightarrow Z$ of $X$ (which is an almost holomorphic map whose general fibers are rationally connected such that $Z$ is not uniruled). Since $X$ is not rationally connected, $\dim Z\geq 1$. On the other hand, as $X'$ is fibered by rationally connected surfaces and $Z$ is not uniruled, we have $\dim Z\leq 1$. But then $h^1(X,\mathcal O_X)\geq h^1(Z,\mathcal O_Z)\geq 1$, which implies that $X$ is smooth by Lemma~\ref{lem:q>0 smooth}, a contradiction.

If $\dim Y=3$, then $X'\simeq Y$ and $K_{X'}\equiv 0$. It follows that $K_{X}\equiv 0$, a contradiction.
\end{proof}

\begin{thm}\label{thm:X=X'/G}
Let $X$ be a terminal projective $3$-fold such that $-K_X$ is nef and $K\not \equiv 0$. Assume that $X$ is singular and not rationally connected. 

Then there exists a finite cover $\pi\colon X'\to X$ \'etale in codimension $2$ which is Galois (in the sense of \cite{matsumura-wang}*{Definition~2.11}) and a morphism $f\colon X'\to Y$ satisfying the following properties:

\begin{enumerate}
 \item One of the following holds:
\begin{enumerate}
 \item $X'$ is a $\mathbb{P}^1$-bundle over $Y$ where $Y$ is an abelian surface, or 
 \item $X'=\mathbb{P}^1\times Y$ where $Y$ is a K3 surface. 
\end{enumerate} 

\item Denote by $G$ the Galois group of $\pi$, then $G$ acts faithfully on $Y$ and
there exists a commutative diagram 
\[
\xymatrix{
X'\ar[d]_{f}\ar[r]^-{\pi}& X=X'/G\ar[d]^-{h}\\ 
Y \ar[r]_-{\mu}& S\coloneqq Y/G
}
\]
such that $\mu$ is \'etale in codimension $1$ and $S$ is a surface with Du Val singularities of type A with $K_S\sim_\mathbb{Q} 0$ and $h^1(S,\mathcal O_S)=0$. 
In particular, the minimal resolution of $S$ is either a K3 surface or an Enriques surface. Furthermore, $\chi(\mathcal{O}_X)=\chi(\mathcal{O}_S)=2$ or $1$ respectively.

\item In Case (a), $c_1(X)\cdot c_2(X)=0$. 
If the minimal resolution of $S$ is a K3 surface, then $G$ is a cyclic group $C_n$ for some $n\in \{2,3,4,6\}$;
if the minimal resolution of $S$ is an Enriques surface, then 
$G\in \{C_2\times C_2, D_8\}$.

\item In Case (b), $G$ acts diagonally on $\mathbb{P}^1\times S$. If the minimal resolution of $S$ is a K3 surface, then $G$ is classified in Table~\ref{tab4}, and in particular $c_1(X)\cdot c_2(X)\geq \frac{4}{5}$.
If the minimal resolution of $S$ is an Enriques surface, then there exists a normal subgroup $G_s$ of $G$ of index $2$ which is classified in Table~\ref{tab5}, and in particular $c_1(X)\cdot c_2(X)\geq 2$.

{
\begin{longtable}{CCCCCC}
\caption{$G$ for K3}\label{tab4}\\
\hline
\# & G & |G| & \text{\rm Sing}(S) & \mathcal{R}_X & c_1c_2 \\
\hline
\endfirsthead
\multicolumn{4}{l}{{ {\bf \tablename\ \thetable{}} \textrm{-- continued}}}
\\
\hline 
\# & G & |G| & \text{\rm Sing}(S) & \mathcal{R}_X & c_1c_2 \\
\endhead
\hline
\hline \multicolumn{4}{c}{{\textrm{Continued on next page}}} \\ \hline
\endfoot

\hline \hline
\endlastfoot
1 & C_2 & 2 & 8A_1 & \{2^{16}\} & 24 \\ 
2 & C_3 & 3 & 6A_2 & \{3^{12}\} & 16 \\ 
3 & D_4 & 4 & 12A_1 & \{2^{24}\} & 12 \\ 
4 & C_4 & 4 & 4A_3, 2A_1 & \{2^{4}, 4^8\} & 12 \\ 
5 & C_5 & 5 & 4A_4 & \{5^{8}\} & 48/5\\ 
6 & D_6 & 6 & 3A_2, 8A_1 & \{2^{16}, 3^6\} & 8 \\ 
7 & C_6 & 6 & 2A_5, 2A_2, 2A_1 & \{2^{4}, 3^4, 6^4\} & 8 \\ 
8 & C_7 & 7 & 3A_6 & \{7^{6}\} & 48/7 \\ 
10 & D_8 & 8 & 2A_3, 9A_1 & \{2^{18}, 4^4\} & 6 \\ 
14 & C_8 & 8 & 2A_7, A_3, A_1 & \{2^{2}, 4^2, 8^4\} & 6 \\ 
16 & D_{10} & 10 & 2A_4, 8A_1 & \{2^{16}, 5^4\} & 24/5 \\ 
17 & \mathfrak A_4 & 12 & 6A_2, 4A_1 & \{2^{8}, 3^{12}\} & 4 \\ 
18 & D_{12} & 12 & A_5, A_2, 9A_1 & \{2^{18}, 3^2, 6^2\} & 4 \\ 
34 & \mathfrak S_4 & 24 & 2A_3, 3A_2, 5A_1 & \{2^{10}, 3^6, 4^4\} & 2 \\ 
55 & \mathfrak A_5 & 60 & 2A_4, 3A_2, 4A_1 & \{2^{8}, 3^6,5^4\} & 4/5 \\ 
\hline
\end{longtable}
}

{
\begin{longtable}{CCCCCC}
\caption{$G_s$ for Enriques}\label{tab5}\\
\hline
\# & G_s & |G| & \text{\rm Sing}(S) & \mathcal{R}_X & c_1c_2 \\
\hline
\endfirsthead
\multicolumn{4}{l}{{ {\bf \tablename\ \thetable{}} \textrm{-- continued}}}
\\
\hline 
\# & G_s & |G| & \text{\rm Sing}(S) & \mathcal{R}_X & c_1c_2 \\
\endhead
\hline
\hline \multicolumn{4}{c}{{\textrm{Continued on next page}}} \\ \hline
\endfoot

\hline \hline
\endlastfoot
1' & C_2 & 4 & 4A_1 & \{2^{8}\} & 12 \\ 
2' & C_3 & 6 & 3A_2 & \{3^{6}\} & 8 \\ 
3' & D_4 & 8 & 6A_1 & \{2^{12}\} & 6 \\ 
4' & C_4 & 8 & 2A_3, A_1 & \{2^{2}, 4^4\} & 6 \\ 
5' & C_5 & 10 & 2A_4 & \{5^{4}\} & 24/5\\ 
7' & C_6 & 12 & A_5, A_2, A_1 & \{2^{2}, 3^2, 6^2\} & 4 \\ 
 16' & D_{10} & 20 & A_4, 4A_1 & \{2^{8}, 5^2\} & 12/5 \\ 
17' & \mathfrak A_4 & 24 & 3A_2, 2A_1 & \{2^{4}, 3^6\} & 2 \\
\hline
\end{longtable}
}
\end{enumerate}

\end{thm}

\begin{proof}
(1) By Theorem~\ref{thm.structure}, there exists a finite cover $\pi\colon X'\to X$ \'etale in codimension $2$ such that 
$X'$ admits a locally constant fibration $f\colon X'\to Y$,
where $Y$ is a terminal projective variety with $K_Y\sim_{\mathbb Q} 0$ and any fiber $F$ of $f$ is rationally connected. By Lemma~\ref{lem:nonrc-Y-not-curve}, $\dim Y= 2$. 
In this case, $Y$ is smooth and so is $X'$. 

By Theorem~\ref{thm.smooth}, after replacing $Y$ by an \'etale cover if necessary, we may assume that either $X'$ is a $\mathbb{P}^1$-bundle over an abelian surface or $X'=\mathbb{P}^1\times Y$ where $Y$ is a K3 surface. 
Note that $X'\to X$ might not be Galois as we required.

Let $\gamma\colon W\to X'$ be the Galois closure of $\pi$ (see \cite{matsumura-wang}*{Definition~2.11} and \cite{gkp}*{Theorem 3.7}). Since $\pi$ is \'etale in codimension $2$, $\gamma$ is also \'etale in codimension $2$, and hence \'etale as $X'$ is smooth. 
In particular, $W$ is also smooth. 

If $X'$ is a $\mathbb{P}^1$-bundle over an abelian surface, then $-K_W=\gamma^*(-K_{X'})$ is nef and we have $h^1(W,\mathcal O_W)\geq h^1(X',\mathcal O_{X'})=2$. So we may apply Theorem~\ref{thm.smooth} to $W$ to conclude that $W$ is again a $\mathbb{P}^1$-bundle over an abelian surface. We may replace $\pi\colon X'\to X$ by $ \pi\circ\gamma\colon W\to X$ to get the desired Galois finite cover.

If $X'=\mathbb P^1\times Y$ where $Y$ is a K3 surface, then $W=X'$ as $X'$ is simply connected. So $\pi$ is the desired Galois finite cover.

(2) Let $g\in G$ be an automorphism of $X'$. Note that $f$ is exactly the MRC fibration of $X'$. 
By the (birational) functoriality of MRC fibrations (see \cite{kollar-book}*{5.5 Theorem}), $g$ induces a birational map $g_Y\colon Y\dasharrow Y$. Since $Y$ is a minimal surface, $g_Y$ is an automorphism.
So $G$ naturally acts on $Y$. 
Denote $S\coloneqq Y/G$ and $\mu\colon Y\to S$ the quotient map. There is a commutative diagram
\[
\xymatrix{
X'\ar[d]_{f}\ar[r]^-{\pi}& X=X'/G\ar[d]^-{h}\\ 
Y \ar[r]_-{\mu}& S\coloneqq Y/G.
}
\]
As $\pi$ is \'etale in codimension $2$, for any general fiber $F\simeq \mathbb P^1$ of $f$, 
$g|_F \colon F\to g(F)\simeq\mathbb P^1$ is \'etale, hence is an identity by the simple connectedness of $\mathbb P^1$. 
Therefore, $\deg \mu=\deg \pi=|G|$ and the action of $G$ on $Y$ is faithful.

By construction, $-K_X$ is ample over $S$ as $-K_{X'}$ is ample over $Y$, so $h\colon X\to S$ is a $\mathbb Q$-conic bundle in the sense of \cite{mp}*{(1.1) Definition}, 
hence $S$ has only Du Val singularities 
of type $A$ by \cite{mp}*{(1.2.7) Theorem}.
Moreover, $R^ih_*\mathcal O_X=0$ for $i>0$ by \cite{mp}*{(2.3) Theorem}
and hence $\chi(\mathcal{O}_X)=\chi(\mathcal{O}_S)$.

By Lemma~\ref{lem:q>0 smooth}, $h^1(X,\mathcal O_X)=0$, which implies that $h^1(S,\mathcal O_S)=0$.
On the other hand, $0\sim K_Y=\mu^*(K_S+R)$ where $R$ is the ramification divisor. If $K_S$ is not pseudo-effective, then $S$ is birational to a ruled surface over $\mathbb{P}^1$ (as $h^1(S,\mathcal O_S)=0$) by the minimal model program, which implies that $S$ is rational. But this implies that $X$ is rationally connected by \cite{ghs}*{Corollary 1.3}, a contradiction. So $K_S$ is pseudo-effective, which implies that $K_S\equiv 0$ and $R=0$. 
The minimal resolution $S'$ of $S$ is a smooth projective surface with $K_{S'}\equiv 0$ and $h^1(S',\mathcal O_{S'})=0$, so it is either a K3 surface or an Enriques surface. Moreover, we can also conclude that $\mu$ is \'etale in codimension $1$ by $R=0$.

(3) Suppose that $X'$ is a $\mathbb{P}^1$-bundle over $Y$ where $Y$ is an abelian surface. Then $c_1(X)\cdot c_2(X)=0$ by Theorem~\ref{thm.smooth} and Lemma~\ref{lem.chern}.

We recall the well-known results of quotients of abelian surfaces from \cites{fujiki,wendland, yoshihara} and the references therein. 
There exists an exact sequence of groups
\[
1\to G_0\to G\to H\to 1
\]
where $G_0$ is the maximal translation subgroup. Then 
there is a commutative diagram
\[
\xymatrix{
X'\ar[d]\ar[r]& X'/G_0\ar[d]\ar[r] & X\ar[d]\\ 
Y \ar[r]& Y/G_0 \ar[r]&S.}
\]
Here $Y/G_0$ is an abelian surface.
Since $X'\to X$ is \'etale in codimension $2$, $X' \to X'/G_0$ is also \'etale in codimension $2$. 
This implies that $-K_{X'/G_0}$ is nef and
$X'/G_0$ is terminal by \cite{kollar-mori}*{Proposition 5.20}. On the other hand, $h^1(X'/G_0,\mathcal O_{X'/G_0})\geq h^1(Y/G_0,\mathcal O_{Y/G_0})=2$, so $X'/G_0$ is smooth by Lemma~\ref{lem:q>0 smooth}. By Theorem~\ref{thm.smooth}, $X'/G_0\to Y/G_0$ is a $\mathbb{P}^1$-bundle. So we may replace $X'\to Y$ by $X'/G_0\to Y/G_0$ and assume further that $H=G$, namely, $G$ does not contain translations.

If the minimal resolution of $S$ is a K3 surface, then $G=C_n$ for some $n\in \{2,3,4,6\}$ by \cite{wendland}*{(2,2)} or \cite{garb}*{Proposition~4.3}, as $S$ has only Du Val singularities of type A.

If the minimal resolution of $S$ is an Enriques surface, then $G\in \{C_2\times C_2, D_8\}$ 
 by \cite{yoshihara}*{Theorem~2.1}. 


(4) Suppose that $X'=\mathbb{P}^1\times Y$ where $Y$ is a K3 surface.

Since $\mathbb{P}^1$ is the anti-canonical model of $X'$, the action of $G$ on $X'$ naturally induces an action on $\mathbb{P}^1$.
There is a commutative diagram
\[
\xymatrix{
X'\ar[d]_{p'}\ar[r]^-{\pi}& X=X'/G\ar[d]^-{p}\\ 
\mathbb{P}^1 \ar[r]_-{\tau}& \mathbb{P}^1=\mathbb{P}^1/G.
}
\]
It is clear that $G$ acts diagonally on $\mathbb{P}^1\times Y$ (but the action might not be faithful on $\mathbb{P}^1$).

\begin{case}
Suppose that the minimal resolution of $S$ is a K3 surface.

Take a general fiber $T$ of $p$ and take $T'\simeq Y$ to be a fiber of $p'$ mapping to $T$, then $X'\to X\to S$ induces finite maps
$T'\to T\to S$. 
By (2),
$T'\to S$ (which is just $\mu$) is \'etale in codimension $1$. Then $T'\to T$ is \'etale
as $T$ is smooth. Hence $T$ can only be a K3 surface or an Enriques surface. Since the minimal resolution of $S$ is a K3 surface, $K_S\sim 0$ and hence $T$ is a K3 surface and $T'\to T$ is an isomorphism. This implies that $|G|=\deg \pi=\deg \tau$ and $G$ acts on $\mathbb{P}^1$ faithfully. In particular, $G$ is a finite subgroup of $\text{PGL}_2(\mathbb{C})$, which can only be one of the following groups:
\begin{itemize}
 \item a cyclic group $C_n$, 
 \item a dihedral group $D_{2n}$, 
 \item the tetrahedral group $\mathfrak A_4$, 
 \item the octahedral group $\mathfrak S_4$, or 
 \item the icosahedral group $\mathfrak A_5$.
\end{itemize}
On the other hand, Xiao \cite{xiao}*{Theorem 3} lists all possible groups $G$ acting on a K3 surface $Y$ such that the minimal resolution of $Y/G$ is a K3 surface. So we get all possible $G$ with corresponding $\text{\rm Sing}(S)$ as listed in Table~\ref{tab4}. Then $c_1(X)\cdot c_2(X)=\frac{c_1(X')\cdot c_2(X')}{|G|}=\frac{48}{|G|}$ by Lemma~\ref{lem.chern}. The computation of $\mathcal{R}_X$ will be explained in Example~\ref{ex.3}.
\end{case}

\begin{case}
Suppose that the minimal resolution of $S$ is an Enriques surface.

Then there exists an exact sequence of groups induced by the action on $\mathbb{C}\simeq H^0(K_Y)$:
\[
1\to G_s\to G\to N\to 1
\]
where $G_s$ consists of symplectic automorphisms and $N$ is a cyclic group. Then 
there is a commutative diagram
\[
\xymatrix{
X'\ar[d]\ar[r]& X'/G_s\ar[d]\ar[r] & X\ar[d]\\ 
Y \ar[r]& Y/G_s \ar[r]&S.}
\]
As $K_{Y/G_s}\sim 0$ and $2K_S\sim 0$, we have that $N=C_2$ and $Y/G_s\to S$ is an \'etale double cover induced by the $2$-torsion Cartier divisor $K_S$. Note that the minimal resolution of $Y/G_s$ is a K3 surface, hence $X'/G_s$ satisfies Case (b) of the theorem. Therefore, $G_s$ and $\text{\rm Sing}(Y/G_s)$ are listed in Table~\ref{tab4}. As $Y/G_s\to S$ is an \'etale double cover, we may rule out the cases in which singularities of $Y/G_s$ appear not in couples, and $\text{\rm Sing}(S)$ (resp. $\mathcal{R}_X$) is the half of $\text{\rm Sing}(Y/G_s)$ (resp. $\mathcal{R}_{X'/G_s}$).\qedhere
\end{case}\setcounter{case}{0}
\end{proof}


\begin{rem}
Conversely, if $X$ admits a non-trivial finite cover $\pi\colon X'\to X$ \'etale in codimension 2 where $X'$ is either
a $\mathbb P^1$-bundle over an abelian surface $Y$ or a product of $\mathbb P^1$ with a K3 surface $Y$,
then similar arguments as in Theorem~\ref{thm:X=X'/G} show that $X$ is not rationally connected. Indeed, 
as in the proof of Theorem~\ref{thm:X=X'/G}(1)(2),
after replacing $\pi$ by its Galois closure, we can assume that $X=X'/G$ for some Galois group $G$; by the functoriality of MRC fibrations,
there exists an induced fibration $X\to S:=Y/G$, where $Y$ is the abelian surface or the K3 surface. If $X$ is rationally connected, then $S$ is also rationally connected.
It follows that the ramified locus $R$ of $\mu\colon Y\to S$ on $Y$ has dimension $1$. 
However, as $\pi$ is \'etale in codimension 2 and $\mathbb P^1$ is simple connected,  
the induced action of $G$ is faithful on generic points of irreducible components of $R$ as the first paragraph of the proof of    
Theorem~\ref{thm:X=X'/G}(2), which contradicts to the fact that $R$ is the ramified locus.
\end{rem}

In the rest of this section,
we exhibit concrete examples satisfying Theorem~\ref{thm:X=X'/G} with $\chi(\mathcal O_X)=2$.

\begin{ex}\label{ex.3}
Let $G\subset \text{PGL}_2(\mathbb{C})$ be a finite group. 
Let $Y$ be an abelian surface or a K3 surface with a faithful $G$ action, such that 
$S\coloneqq Y/G$ has only Du Val singularities of type A and the minimal resolution of $S$ is a K3 surface.
Then consider
$X\coloneqq (\mathbb P^1\times Y)/G$,
where $G$ acts diagonally on $\mathbb P^1\times Y$. 

We will check that $X$ is a terminal projective $3$-fold such that $-K_X$ is nef and $K_X\not \equiv 0$ satisfying Theorem~\ref{thm:X=X'/G} with $X'=\mathbb P^1\times Y$ and $\chi(\mathcal{O}_X)=2$.

First we check that $X$ has terminal singularities. It is clear that over the smooth locus of $S$, the quotient maps are \'etale and hence $X$ is smooth. 
On the other hand, let $P\in S$ be a singular point of type $A_n$. By Lemma~\ref{lem:G-cyclic}, there exists an  analytic  open neighborhood $U\subset Y$ over a neighborhood of $P$ such that
locally $X= (\mathbb P^1\times U)/C_{n+1}$, which implies that $X$ has $2$ terminal singularities of type $\frac{1}{n+1}(\pm 1,1,-1)$ locally over $P\in S$. In particular, each singular point of type $A_n$ on $S$ contributes two cyclic quotient singularities of index $n+1$ on $X$, so we can gather all those singularities to form $\mathcal{R}_X$.
This local computation also shows that $ \mathbb P^1\times Y \to X$ is \'etale in codimension $2$ and hence $-K_X$ is nef. Finally, $\chi(\mathcal{O}_X)=\chi(\mathcal{O}_S)=2$ by Theorem~\ref{thm:X=X'/G}(2).

This example shows that all groups $G$ in Theorem~\ref{thm:X=X'/G}(3)(4) can be realized when the minimal resolution of $S$ is a K3 surface. 
For the existence of surfaces $Y$ with certain group actions we refer to \cite{fujiki}, \cite{wendland}*{(2.2)}, \cite{garb}*{Proposition~4.3}, \cite{nikulin}*{Theorem 4.5}, \cite{xiao}*{Theorem 3}, \cite{mukai}*{(0.4) Example}.
In particular, when $Y$ is a K3 surface and $G=\mathfrak A_5$, we get an example of a terminal projective $3$-fold $X$ with $c_1(X)\cdot c_2(X)=\frac{4}{5}$ attaining the minimal value in Theorem~\ref{thm:X=X'/G}.
\end{ex}

\begin{lem}\label{lem:G-cyclic}
 Let $Y$ be a smooth surface with a finite group $G$ acting freely in codimension $1$. Suppose that $P\in Y/G$ is a 
 Du Val singularity of type $A_n$ and $Q\in Y$
is a point in its preimage. Then the stabilizer of $Q$ is a subgroup of $G$ isomorphic to $C_{n+1}$ and there exists an analytic open neighborhood $U$ of $Q$ such that $Y/G\simeq U/C_{n+1}$ in a neighborhood of $P$.
\end{lem}
\begin{proof}
As $P$ is of type $A_n$, the stabilizer of $Q$ is isomorphic to $C_{n+1}$ (see for example \cite{xiao}*{Page~74}). Other parts are clear.
\end{proof}



\section{The case when $-K$ is big}
In this section, we consider the case when $X$ is a weak Fano $3$-fold.

 \begin{ex}\label{ex.fano}
 There are many known examples of Fano $3$-folds, for example, Iano-Fletcher's list \cite{IF00}*{\S 16.6, \S 16.7} of $\mathbb{Q}$-factorial terminal Fano $3$-folds of Picard number $1$ and Kasprzyk's classification \cite{kasprzyk} of terminal toric Fano $3$-folds (which can also be found in http://www.grdb.co.uk/search/toricf3t). 
 The baskets of all these examples are known. 
 By \eqref{eq.range}, we can compute that for all these examples, $c_1 c_2 >2$, where the minimal one is obtained by \cite{IF00}*{\S 16.7 No.85} with $c_1 c_2 =\frac{431}{180}$.
 \end{ex}

\begin{thm}\label{thm.fano}
Let $X$ be a $\mathbb{Q}$-factorial klt weak Fano variety of dimension $n$. Assume that $X$ is smooth in codimension $2$.
Then $c_1(X)^{n-2}\cdot c_2(X)>0$.
\end{thm}

\begin{proof}
Let $\mathcal{T}_{X}$ be the tangent sheaf of $X$ 
and $H$ be an ample Cartier divisor on $X$.
Set $H_{ \varepsilon }:=-K_X+\varepsilon H$ for any $\varepsilon>0$.
By \cite{matsuki}*{Proposition 6.4} (see also \cite{Cao13}*{Proposition 2.3}), if $\varepsilon>0$ is small enough,
then the $(H_{\varepsilon}^{n-1})$-Harder--Narasimhan filtration 
\[
0 =: \mathcal{E}_0 \subset \mathcal{E}_1 \subset \cdots \subset \mathcal{E}_l:=\mathcal{T}_{X}
\]
is independent of $\varepsilon$.
Set $\mathcal{G}_i:= \mathcal{E}_i / \mathcal{E}_{i-1}$ and 
$r_i := \rank \mathcal{G}_i$.
By \cite{ou}*{Theorem 1.3}, 
$\mathcal{T}_{X}$ is $H_{ \varepsilon }^{n-1}$-generically nef, 
in particular, $\mu_{H_{ \varepsilon }}(\mathcal{G}_1) > \cdots > \mu_{H_{ \varepsilon }}(\mathcal{G}_l)\ge0$.
Since the sheaf $\mathcal{G}_i$ is an $H_{ \varepsilon }$-semistable sheaf,
the Bogomolov--Gieseker inequality yields
\begin{align}
\begin{split}
\label{first_inequality}
2c_2(\mathcal{T}_{X}) H_{ \varepsilon }^{n-2}
&= \left(c_1(\mathcal{T}_{X})^2 + \sum_{1\le i\le l}(2c_2(\mathcal{G}_i) -c_1(\mathcal{G}_i)^2)\right) H_{ \varepsilon } ^{n-2} \\
&\ge \left(c_1(\mathcal{T}_{X})^2 - \sum_{1\le i\le l}\frac{c_1(\mathcal{G}_i)^2}{r_i}\right) H_{ \varepsilon } ^{n-2}. \\
\end{split}
\end{align}
Set $a_i := c_1(\mathcal{G}_i) c_1(\mathcal{T}_{X})^{n -1}$.
Then, we obtain that
$$
\sum_{1 \le i \le l } a_i = c_1(\mathcal{T}_{X})^n > 0
\quad \text{and} \quad
\mu_{H_{ \varepsilon }}(\mathcal{G}_i)= \frac{c_1(\mathcal{G}_i) H_{ \varepsilon }^{n-1}}{r_i}=\frac{a_i}{r_i} +O(\varepsilon).
$$
From $\mu_{H_{ \varepsilon }}(\mathcal{G}_1) > \cdots > \mu_{H_{ \varepsilon }}(\mathcal{G}_l)\ge0$,
we obtain $a_1/r_1 \ge a_2/r_2 \ge \cdots \ge a_l/r_l\ge0$
for sufficiently small $ \varepsilon >0$.

\begin{claim}
\label{nu2_inequality}The following estimate holds:
\[
2c_2(\mathcal{T}_X) H_{ \varepsilon }^{n-2}
\ge
 \frac{(r_1 - 1)a_{1}}{r_1} + \sum_{2 \le j\le l} a_j + O( \varepsilon ).
\]
In particular, 
if $a_2>0$ or $r_1 \neq 1$, then 
$c_2(\mathcal{T}_{X}) c_1(\mathcal{T}_{X})^{n-2}>0$.
\end{claim}

\begin{proof}[Proof of Claim~\ref{nu2_inequality}]
This argument is the same as that in \cite{IM22}*{Appendix}.
By the Hodge index theorem, we obtain
\begin{equation}
c_1(\mathcal{G}_i)^2H_{ \varepsilon } ^{n-2} 
\le \frac{(c_1(\mathcal{G}_i) c_1(\mathcal{T}_X)H_{ \varepsilon }^{n-2})^2}{c_1(\mathcal{T}_X)^2H_{ \varepsilon }^{n-2}}.
\end{equation}
We also obtain the following equalities
\begin{align}
\begin{split}
\label{Hodge_index}
\left(c_1(\mathcal{T}_X)^2 H_{ \varepsilon }^{n-2}\right)^{-1} 
&=
\left( \sum_{1\le k\le l} a_k + O( \varepsilon) \right)^{-1}=
\left(\sum_{1\le k\le l} a_k \right)^{-1}+ O( \varepsilon ). 
\end{split}
\end{align}
From \eqref{first_inequality}-\eqref{Hodge_index}, we obtain

\begin{align*}
\begin{split}
2c_2(\mathcal{T}_X) H_{ \varepsilon } ^{n-2}
&\ge c_1(\mathcal{T}_X)^2H_{ \varepsilon }^{n-2} - \sum_{1\le i\le l}\frac{(c_1(\mathcal{G}_i)c_1(\mathcal{T}_X) H_{ \varepsilon }^{n-2})^2}{r_ic_1(\mathcal{T}_X)^2 H_{ \varepsilon }^{n-2}} \\
&=\frac{1}{c_1(\mathcal{T}_X)^2 H_{ \varepsilon }^{n-2}}\left\{ \left( c_1(\mathcal{T}_X)^2 H_{ \varepsilon }^{n-2} \right)^{2} - \sum_{1\le i\le l}\frac{1}{r_i}(c_1(\mathcal{G}_i)c_1(\mathcal{T}_X) H_{ \varepsilon }^{n-2})^2
 \right\} \\
&= \left(\sum_{1\le k\le l} a_k \right)^{-1}
\left(\sum_{1\le i\le l} a_i \sum_{1\le j\le l} a_j -\sum_{1\le i\le l} \frac{a_{i}^{2}}{r_i}\right)
+ O( \varepsilon) \\
&\ge \left(\sum_{1\le k\le l} a_k \right)^{-1}
\left(\sum_{1\le i\le l} a_i \sum_{1\le j\le l} a_j -\frac{a_1}{r_1}\sum_{1\le i\le l} a_i \right)
+ O( \varepsilon) \\
&= \frac{(r_1 - 1)a_{1}}{r_1} + \sum_{2 \le j\le l} a_j + O( \varepsilon ).\\
\end{split}
\end{align*}
If $a_2>0$ or $r_1 \neq 1$, then by taking $\varepsilon \rightarrow 0$, we obtain
\[
2c_2(\mathcal{T}_X) c_1(\mathcal{T}_X)^{n-2} 
\ge 
\frac{(r_1 - 1)a_{1}}{r_1} + \sum_{2 \le j\le l} a_j >0. \qedhere
\] 
\end{proof}

Therefore, it suffices to show that $a_2>0$ or $r_1 \neq 1$ holds.
For a contradiction, we assume that $a_2=0$ and $r_1 = 1$. Note that $l>1$.
Set $\mathcal{F}:=\mathcal{E}_1 \subset \mathcal{T}_{X}$.
From $a_2 =0$, 
\begin{equation}
\label{1.3}
c_1(\mathcal{F}) c_1(\mathcal{T}_X)^{n-1} = c_1(\mathcal{T}_X)^{n} >0 .
\end{equation}
Since then $c_1(\mathcal{F}) H_{\varepsilon}^{n-1} >0$, 
taking a general curve $C$ representing $(mH_{\varepsilon})^{n-1}$ for $m\gg 1$, we have that
$\mathcal{F}|_C$ is ample, hence $\mathcal{F}$ is an algebraically integrable foliation with rationally connected leaves by \cite{kct}*{Theorem 1}. 
Let $f\colon X \dashrightarrow Y$ be a dominant rational map induced by $\mathcal{F}$. 
Notice that the Zariski closure of a general fiber of $f$ is a rational curve since $\rank \mathcal{F}=1$ and thus $\dim Y = \dim X - 1 \ge 1$.
Take a resolution $\pi\colon \tilde X \to X$ of the indeterminacy locus of $f$. 
Here we have the following commutative diagram:
\begin{equation*}
\xymatrix@C=40pt@R=30pt{
\tilde{X} \ar@{->}[rd]^{\tilde{f}} \ar@{->}[d]^{\pi}&\\ 
X \ar@{-->}[r]^{f}&Y.\\ 
 }
 \end{equation*}
 We may assume that both $\tilde{X}$ and $Y$ are smooth.
 Since $X$ is klt,
 there exist $\pi$-exceptional $\mathbb Q$-effective divisors 
 $D$ and $G$ without common component on $\tilde{X}$ such that 
\[
K_{\tilde{X}} + D\sim_{\mathbb Q} \pi^{*}K_X + G
\]
and $(\tilde{X}, D)$ is a klt pair.

\begin{claim}
\label{pseudo-effective}
 Let $A_Y$ be an ample Cartier divisor on $Y$ and $\tilde{\mathcal{F}}$ be the foliation induced by $\tilde{f}$.
 Then, there exists $\delta >0$ such that 
 $$
 K_{\tilde{\mathcal{F}}} - \pi^{*} K_{X} + D - \delta \tilde{f}^{*}A_Y
 $$
 is pseudo-effective.
\end{claim}

\begin{proof}[Proof of Claim~\ref{pseudo-effective}]
The proof is almost the same as that in \cite{ou}*{Theorem 1.10}.
Since $-\pi^{*}K_{X}$ is nef and big, there exists an ample $\mathbb Q$-divisor $A_{\tilde{X}}$ and an effective $\mathbb Q$-divisor $D'$ on $\tilde{X}$ such that
$$
- \pi^{*}K_{X} - A_{\tilde{X}} \sim_{\mathbb Q} D'
 $$
 and $(\tilde{X}, D + D')$ is a klt pair.
 By taking $\delta >0$ small enough, $A_{\tilde{X}} -\delta \tilde{f}^{*}A_Y$ is ample, so we may assume that $A_{\tilde{X}} -\delta \tilde{f}^{*}A_Y$ is effective and sufficiently general. 
Set $\Delta : = D + D' + (A_{\tilde{X}} -\delta \tilde{f}^{*}A_Y) $.
Hence, $(\tilde{X} , \Delta)$ is klt and 
 $$
 K_{\tilde{\mathcal{F}}} + \Delta
 \sim_{\mathbb Q}
 K_{\tilde{\mathcal{F}}} - \pi^{*} K_{X} + D - \delta \tilde{f}^{*}A_Y.
 $$
For any general fiber $\tilde{F}$ of $\tilde{f}$, 
$$
(K_{\tilde{X}} + \Delta)|_{\tilde{F}}
\sim_{\mathbb Q}
(G - \delta \tilde{f}^{*}A_Y)|_{\tilde{F}}
\sim_{\mathbb Q} G|_{\tilde{F}}.
$$
In particular, 
$\kappa(\tilde{F}, (K_{\tilde{X}} + \Delta)|_{\tilde{F}})\geq 0$. 
Hence, by \cite{Dru17}*{Proposition 4.1}, 
$K_{\tilde{\mathcal{F}}} + \Delta$ is pseudo-effective.
\end{proof}
 
 Since both $ K_{\tilde{\mathcal{F}}}-\pi^{*}K_{\mathcal{F}}$ and $D$ are $\pi$-exceptional,
 by Claim~\ref{pseudo-effective}, we obtain 
\[
(K_{\mathcal{F}}-K_{X} )\cdot (-K_{X})^{n-1} - 
 \delta \tilde{f}^{*}A_Y \cdot \pi^{*}(-K_{X})^{n-1} 
\ge 0.
\]
From (\ref{1.3}), we have $0 \ge \tilde{f}^{*}A_Y \cdot \pi^{*}(-K_{X})^{n-1}$.\
On the other hand,
since $-\pi^{*}K_{X}$ is nef and big, 
$-\pi^{*}K_{X}\sim_{\mathbb Q} A + E$ holds for some ample $\mathbb Q$-divisor $A$ and some effective $\mathbb Q$-divisor $E$, hence we have 
\begin{align*}
\begin{split}
\tilde{f}^{*}A_Y \cdot\pi^{*}(-K_{X})^{n-1} 
&= \tilde{f}^{*}A_Y \cdot \pi^{*}(-K_{X})^{n-2} \cdot(A + E) \\
&\ge \tilde{f}^{*}A_Y \cdot \pi^{*}(-K_{X})^{n-2} \cdot A
\ge \cdots \ge \tilde{f}^{*}A_Y \cdot A^{n-1} >0 .
\end{split}
\end{align*}
This is a contradiction.
\end{proof}


We can remove the assumption on $\mathbb{Q}$-factoriality for terminal weak Fano varieties. 
\begin{cor}\label{cor.fano}
Let $X$ be a terminal weak Fano variety of dimension $n$. 
Then $c_1(X)^{n-2}\cdot c_2(X)>0$.
\end{cor}

\begin{proof}
 Let $\rho \colon Z\to X$ be a $\mathbb Q$-factorialization of $X$ (see \cite{bchm}*{Corollary 1.4.3} for the existence of a $\mathbb Q$-factorialization).
Note that $\rho$ is small. 
Hence $Z$ is again terminal, and smooth in codimension $2$.

By definition, $c_1(Z)=c_1(\rho^{-1}(X_0))=\rho^*c_1(X_0)=\rho^*c_1(X)$ since $\rho$ is small, 
and $c_2(Z)=c_2(\rho^{-1}(X_0))+B$ where $B$ is an element of the second Chow group of $Z$, supporting on the exceptional locus of $\rho$.
It follows that $\rho_*c_2(Z)=c_2(X_0)=c_2(X)$ and $c_1(Z)^{n-2}\cdot c_2(Z)=\rho^*(c_1(X)^{n-2})\cdot c_2(Z)=c_1(X)^{n-2}\cdot \rho_*c_2(Z)=c_1(X)^{n-2}\cdot c_2(X)$ by the projection formula. 
Therefore, the conclusion follows by applying Theorem~\ref{thm.fano} to $Z$.
\end{proof}

Combining with BAB theorem proved by Birkar, we can obtain an inequality similar to \cite{kawamata92}*{Proposition 1}:

\begin{cor}\label{cor.chernbound}
There exists a positive rational number $b$ depending only
on $n$ such that the inequality 
\[
c_1(X)^n\leq b \cdot c_1(X)^{n-2}\cdot c_2(X)
\]
holds for any terminal weak Fano variety $X$ of dimension $n$.
\end{cor}

\begin{proof}
Fix a positive integer $n$.
By \cite{birkar}, the set of terminal weak Fano varieties $X$ of dimension $n$ is bounded. On the other hand, Corollary~\ref{cor.fano} implies that $c_1(X)^{n-2}\cdot c_2(X)>0$. Hence the function $\frac{c_1(X)^n}{c_1(X)^{n-2}\cdot c_2(X)}$ 
is well-defined and locally constant in flat families. So it has a uniform upper bound for any terminal weak Fano variety $X$ of dimension $n$. 
\end{proof}

We would like to obtain an effective bound $b$. We obtain a partial result by the same argument of Theorem $\ref{thm.fano}$. 
\begin{cor}
Let $X$ be a $\mathbb{Q}$-factorial klt weak Fano variety of dimension $n$. Assume that $X$ is smooth in codimension $2$.
Then, one of the following holds:
\begin{enumerate}
\item $c_1(X)^n\leq 4c_1(X)^{n-2}\cdot c_2(X)$.
\item There exists a dominant rational map $f\colon X \dashrightarrow Y$ such that the Zariski closure of a general fiber of $f $ is a rational curve.
\end{enumerate}
\end{cor}
\begin{proof}
We use the same notation as in the proof of Theorem~\ref{thm.fano}. 
If $r_1=1$, then by \cite{kct}*{Theorem 1}, we obtain (2). 
So we assume that $r_1\geq 2$.
By the same calculations, we have
\begin{align*}
\begin{split}
4c_2(\mathcal{T}_X) H_{ \varepsilon } ^{n-2} - c_1(\mathcal{T}_{X})^2 H_{ \varepsilon } ^{n-2}
&\ge 2\left(\sum_{1\le k\le l} a_k \right)^{-1}
\left(\sum_{1\le i\le l} a_i \sum_{1\le j\le l} a_j -\sum_{1\le i\le l} \frac{a_{i}^{2}}{r_i}\right)
- \sum_{1\le i\le l} a_i + O( \varepsilon) \\
&=\left(\sum_{1\le k\le l} a_k \right)^{-1}
\left(\sum_{1\le i\le l} a_i \sum_{1\le j\le l} a_j -2\sum_{1\le i\le l} \frac{a_{i}^{2}}{r_i}\right)
+ O( \varepsilon) \\
&\ge \left(\sum_{1\le k\le l} a_k \right)^{-1}
\left(\sum_{1\le i\le l} a_i \sum_{1\le j\le l} a_j -\frac{2a_{1}}{r_1}\sum_{1\le i\le l} a_i\right)
+ O( \varepsilon) \\
&= 
\frac{(r_1- 2)a_{1}}{r_1} + \sum_{2 \le j\le l} a_j + O( \varepsilon ).\\
\end{split}
\end{align*}
Therefore, if $r_1>2$ or $a_2>0$, then by taking $\varepsilon \rightarrow 0$, we obtain 
\[
4c_2(\mathcal{T}_X) c_1(\mathcal{T}_X)^{n-2} - 
c_1(\mathcal{T}_X)^{n}\ge 
\frac{(r_1- 2)a_{1}}{r_1} + \sum_{2 \le j\le l} a_j >0. 
\] 
If $r_1=2$ and $a_2=0$, then by the same argument after Claim \ref{nu2_inequality} of Theorem~\ref{thm.fano},
we have $n=r_1=2$. 

Indeed, if $n > 2$, then from the discussion after Claim~\ref{nu2_inequality} in Theorem~\ref{thm.fano}, 
there exists a dominant rational map $f\colon X \dashrightarrow Y$ such that the Zariski closure of the general fiber is a rationally connected surface. 
From $n > 2$, we have $\dim Y \ge 1$, which is a contradiction by using the same argument after Claim~\ref{pseudo-effective}. 
Hence, we can conclude that $n=r_1=2$. 

In this case, $X$ is a surface and $\mathcal{T}_{X}$ is $H_{\varepsilon}$-semistable, thus by the Bogomolov--Gieseker inequality, we obtain $4c_2(\mathcal{T}_X) \geq c_1(\mathcal{T}_X)^2$. 
\end{proof}

For terminal weak Fano $3$-folds, the upper bound $b$ can be more explicit:

\begin{cor}\label{cor.effbd}
There exists a positive integer $b=2^4\cdot 3^6\cdot 7$ such that
an inequality 
\[
c_1(X)^3\leq b\cdot c_1(X)\cdot c_2(X)
\]
holds for any terminal weak Fano $3$-fold $X$.
\end{cor}
\begin{proof}
It follows directly by the facts that $c_1(X)^3=(-K_X)^3\leq 324$ by \cite{jiang-zou}*{Theorem 1.1} and 
$c_1(X)\cdot c_2(X)\geq \frac{1}{252}$ by Theorem~\ref{thm.main.bound} and Corollary~\ref{cor.fano}.
\end{proof}

\begin{rem}\label{rem.b=3}
For Gorenstein terminal weak Fano $3$-folds, there exists a sharp bound $b=3$. In fact, $c_1(X)\cdot c_2(X)=24$ by Theorem~\ref{thm.euler} and by the fact that $\chi(\mathcal O_X)=1$. On the other hand, $c_1(X)^3\leq 72$ by Prokhorov \cite{prok}. 
\end{rem}

\section{Proof of main theorems}
\begin{proof}[Proof of Theorem~\ref{thm.main.bound}]
Let $X$ be a terminal projective $3$-fold such that $-K_X$ is nef and $c_1(X)\cdot c_2(X)\neq 0$.

If $X$ is rationally connected, then $c_1(X)\cdot c_2(X)\geq \frac{1}{252}$  by Theorem~\ref{thm.list}; moreover, if $-K_X$ is not big, then $ c_1(X)\cdot c_2(X)\geq \frac{2}{5}$. 

If $X$ is smooth, then $c_1(X)\cdot c_2(X)\geq 24$ by Theorem~\ref{thm.smooth}.

If $X$ is singular and not rationally connected, then  $c_1(X)\cdot c_2(X)\geq \frac{4}{5}$ by Theorem~\ref{thm:X=X'/G}. 

These three cases conclude Theorem~\ref{thm.main.bound}.
\end{proof}

\begin{proof}[Proof of Theorem~\ref{thm.roughclassify}]

Let $X$ be a terminal projective $3$-fold such that $-K_X$ is nef and $K_X\not\equiv 0$. 

If $X$ is rationally connected, then we get (2).

If $X$ is smooth and not rationally connected, then $X$ is classified in Theorem~\ref{thm.smooth}, and we get (1).

If $X$ is singular and not rationally connected, then $X$ is classified in Theorem~\ref{thm:X=X'/G}, and we get (3). 
\end{proof}

\begin{proof}[Proof of Corollary~\ref{cor.roughclassify=0}]
   Let $X$ be a terminal projective $3$-fold such that $-K_X$ is nef and $c_1(X)\cdot c_2(X)= 0$.

If $X$ is not rationally connected, then we get (1)(2) from the structure of $X$ in  Theorem~\ref{thm.smooth} and Theorem~\ref{thm:X=X'/G}.

If $X$ is rationally connected, then $-K_X$ is not big by Corollary~\ref{cor.fano}, hence $K_X^3=0$. By Theorem~\ref{thm.list}, $\mathcal{R}_X$ is in Table~\ref{tab1}. On the other hand, as $-K_X$ is not big, we can rule out (5), (6) and (10) in Table~\ref{tab1} by Table~\ref{tab2}.
\end{proof}

\end{document}